\documentclass[a4paper,12pt,reqno]{amsart}

\usepackage{amsmath}
\usepackage{amssymb}
\usepackage{amsfonts}
\usepackage{graphicx}
\usepackage[colorlinks]{hyperref}
\renewcommand\eqref[1]{(\ref{#1})} %Need with hyperref

\graphicspath{ {images/} }
\title[Global existence and blow-up of solutions]{Global existence and blow-up of solutions to the double nonlinear porous medium equation}

\author[Bolys Sabitbek]{Bolys Sabitbek}
\address{ \href{http://analysis-pde.org/bolys-sabitbek/}{Bolys Sabitbek:}
	\endgraf
	School of Mathematical Sciences
	\endgraf Queen Mary University of London
	\endgraf
	United Kingdom
	\endgraf
	and
	\endgraf
	Institute of Mathematics and Mathematical Modeling
	\endgraf
	Almaty,
	Kazakhstan
	\endgraf
	{\it E-mail address} {\rm b.sabitbek@qmul.ac.uk}
}
\author[Berikbol Torebek]{Berikbol Torebek$^*$}
\address{{Berikbol Torebek:}
	\endgraf
	Department of Mathematics: Analysis, Logic and Discrete Mathematics
	\endgraf
	Ghent University
	\endgraf
	Belgium
	\endgraf
	and
	\endgraf
	Institute of Mathematics and Mathematical Modeling
	\endgraf
	Almaty,
	Kazakhstan
	\endgraf
	{\it E-mail address} {\rm berikbol.torebek@ugent.be}
}


\thanks{This research has been funded by the Science Committee of the Ministry of Science and Higher Education of the Republic of Kazakhstan (Grant No. AP19676817). The second author was supported by FWO Odysseus 1 grant G.0H94.18N: Analysis and Partial Differential Equations. No new data was collected or generated during the course of this research.\\
	$^*$ Corresponding author}

\newtheoremstyle{theorem}%name
{10pt}          % space above
{10pt}  % space below
{\sl}  % bofy font
{\parindent}     % ident - empty=no indent,  \parindent= paragraph indent
{\bf}  % thm head font
{. }    % punctuation after thm head
{ }    % space after thm head: `` ``=normal \newline=linebreak
{}     % thm head specification
\theoremstyle{theorem}

\numberwithin{equation}{section}
\theoremstyle{plain}
\newtheorem{thm}{Theorem}[section]

\newtheorem{lem}[thm]{Lemma}

\theoremstyle{definition}
\newtheorem{defn}[thm]{Definition}

\newtheoremstyle{defi}%name
{10pt}          % space above
{10pt}  % space below
{\rm}  % bofy font
{\parindent}     % ident - empty=no indent,  \parindent= paragraph indent
{\bf}  % thm head font
{. }    % punctuation after thm head
{ }    % space after thm head: `` ``=normal \newline=linebreak
{}     % thm head specification
\theoremstyle{defi}

\begin{document}
	\begin{abstract}
		
		In this study, we examine a double nonlinear porous medium equation subject to a novel nonlinearity condition within a bounded domain. First, we introduce the blow-up solution for the problem under consideration for the negative initial energy. By introducing a set of potential wells, we construct invariant sets of solutions for the double nonlinear porous medium equation. For subcritical and critical initial energy scenarios, we derive the global existence and asymptotic behavior of weak solutions, as well as blow-up phenomena occurring within a finite time for the positive solution to the double nonlinear porous medium equation.
	\end{abstract}
	\maketitle
	\tableofcontents
	\section{Introduction}
	\subsection{Setting problem}
	The aim of this paper is to study the global existence of a weak solution and the blow-up phenomena in a finite time frame for a double nonlinear porous medium equation in the context of initial energy states. We are interested in the initial-boundary value problem
	\begin{align}\label{main_eqn_p>2}
	\begin{cases}
	u_t(x,t) - \nabla \cdot ( |\nabla u^m|^{p-2}\nabla u^m) = f(u(x,t)), \,\,\, & (x,t) \in \Omega \times (0,+\infty), \\
	u(x,t)  =0,  \,\,\,& (x,t) \in \partial \Omega \times [0,+\infty), \\
	u(x,0)  = u_0(x)\geq 0,\,\,\, & x \in \overline{\Omega},
	\end{cases}
	\end{align}
	where $\Omega \subset \mathbb{R}^n$ is a bounded domain with a smooth boundary $\partial \Omega$, $m \geq 1$ and $2\leq p< n$. 
	
	The nonlinearity $f(u)$ satisfies the following assumptions:
	\begin{itemize}
		\item[(i)] $f\in C^1$, $f(0)=f'(0)=0$;
		\item[(ii)] $f(u)$ is convex for $u>0$;
		\item[(iii)]  For $p<\alpha \leq \gamma < \frac{pn}{n-p}$, we have
		\begin{equation}
		u^m f(u)\leq \gamma ( F(u)-\sigma),
		\end{equation}
		and
		\begin{equation}\label{eq-1.3}
		\alpha F(u)\leq u^mf(u)+\beta u^{pm} + \alpha \sigma,
		\end{equation}
		where $F(u) = m\int_0^u s^{m-1}f(s) ds $, $\sigma>0$, and $0<\beta\leq \frac{\lambda_{1,p}(\alpha - p)}{p}$ with $m\geq 1$ and $\lambda_{1,p}$ is the first eigenvalue of $p$-Laplacian.
		
	\end{itemize}
	For our convenience, these assumptions on the nonlinearity $f(u)$ denoted as $(H)$.

	\subsection{Literature overview}
	Chung-Choi \cite{Chung-Choi} recently introduced a novel condition
	\begin{equation}\label{cond-1}
	\alpha \int_0^u f(s)ds \leq uf(u) + \beta u^p + \alpha\sigma, \,\, u>0,
	\end{equation}
	which is applicable for blow-up solutions to $p$-Laplace parabolic equation that can be derived from \eqref{main_eqn_p>2} for $m=1$. Their proof of blow-up phenomena is based on the concavity method, initially presented in abstract form by Levine \cite{Levine} and further developed in subsequent works \cite{LP1} and \cite{PhP-06}. It is important to note that the condition \eqref{cond-1} for $p=2$ on the nonlinearity $f(u)$ includes  the following cases: 
	\begin{align}
	(2+\epsilon) F(u) &\leq uf(u), \label{eq-1.4}\\
	(2+\epsilon) F(u) &\leq uf(u) + \sigma\label{eq-1.5}, \\
	(2+\epsilon) F(u) &\leq uf(u) + \beta u^2 + \sigma \label{eq-1.6},
	\end{align}
	where $0<\beta \leq \epsilon\lambda_1 /2$, $\sigma>0$ and $F(u)=\int_0^u f(s)ds$. Philippin and Proytcheva \cite{PhP-06} use condition \eqref{eq-1.4} to demonstrate the blow-up solutions for the semilinear heat equation, which is a particular instance of the abstract condition by Levine and Payne \cite{LP1}.  Subsequently, Bandle and Brunner \cite{Band-Brun} relaxed this condition by \eqref{eq-1.5}. 
	
	Observe that condition \eqref{eq-1.6} broadens \eqref{eq-1.4} and \eqref{eq-1.5}. The distinction between \eqref{eq-1.6} and \eqref{eq-1.5} lies in the dependence of \eqref{eq-1.6} on the domain due to the term $\beta u^2$. The constant $\beta$ relies on the domain via the first eigenvalue $\lambda_1$. If the first eigenvalue $\lambda_1$ is arbitrarily small, then \eqref{eq-1.6} approaches \eqref{eq-1.5}.
	
	Condition \eqref{eq-1.6} can be expressed as
	\begin{equation*}
	\frac{d}{du}\left( \frac{F(u)}{u^{2+\epsilon}} - \left(\frac{\sigma}{2+\epsilon}\right) \frac{1}{u^{2+\epsilon}} - \frac{\beta}{\epsilon} \frac{1}{u^{\epsilon}} \right) \geq 0, \,\,\, u>0. 
	\end{equation*}
	It then becomes apparent that
	\begin{align*}
	&\eqref{eq-1.4} \,\,\, \text{ holds if and only if } \,\, F(u)=u^{2+\epsilon}h_1(u),\\
	&\eqref{eq-1.5} \,\,\, \text{ holds if and only if } \,\, F(u)=u^{2+\epsilon}h_2(u)+b,\\
	&\eqref{eq-1.6} \,\,\, \text{ holds if and only if } \,\, F(u)=u^{2+\epsilon}h_3(u)+au^2+b,
	\end{align*}
	for some positive constants $\epsilon$, $b$, and $a<\lambda_1/2$, where $h_1$, $h_2$, and $h_3$ are nondecreasing functions on $(0,+\infty)$. Note that the constants $\epsilon$, $b$, and $a$ shall be different in each case. We present an example for the case when $q> 2+\epsilon$
	\begin{equation*}
	f(u) = u^{q-1} + \beta u.
	\end{equation*}
	This example meets the requirement specified by condition \eqref{eq-1.6} ($\sigma=0$), but does not fulfill the conditions stated in \eqref{eq-1.4}.

	As far as we know, the potential wells method was first proposed by Sattinger \cite{Sattinger} to study the properties of solutions of semilinear hyperbolic equation and was developed by many authors (\cite{LiuZh1, LiuZh2, PaSa, QSouplet, Souplet, Tsuts1, Tsuts2, Xu1, Xu2}).
	It is worth noting the paper of Payne and Sattinger \cite{PaSa}, devoted to a more detailed description of the potential wells method for the nonlinear wave and heat equations.
	In 2006, Liu-Zhao \cite{LiuZh2} generalized and improved the results of Payne-Sattinger \cite{PaSa}, by obtaining the invariant sets and vacuum isolation of solutions. Later, Xu \cite{Xu1} complemented the result of Liu-Zhao by considering the critical initial energy $E(0) = d.$
	The applications of the potential wells method to the study of various nonlinear evolution equations is still relevant. This is confirmed by the recently published papers devoted to this direction (see for example \cite{Fern, Ishi, Lian, Liu1, Tran, Xu3, Zhou} and references therein).
	
	The porous medium equation is one of the important examples of nonlinear parabolic equations. The physical applications of the porous medium equation describe widely processes involving fluid flow, heat transfer or diffusion, and its other applications in different fields such as mathematical biology, lubrication, boundary layer theory, and etc. There is a huge literature dealing with an existence and nonexistence of solutions to problem \eqref{main_eqn_p>2} for the reaction term $u^p$ on the bounded and unbounded domains, for example, \cite{Band-Brun, Vazq1, Chen-Fila-Guo, Gal-Vaz-97, Gr-Mu-Po-13, Gr-Mu-Pu-1, Gr-Mu-Pu-2, Ia-San-14, LP1, Sam-Gal-Ku-Mik}. By using the concavity method, Schaefer \cite{Sch09} established a condition on the initial data of a Dirichlet type initial-boundary value problem for the porous medium equation with a power function reaction term when blow-up of the solution in finite time occurs and a global existence of the solution holds.  We refer more details to Vazquez's book \cite{Vaz} which provides a systematic presentation of the mathematical theory of the porous medium equation.
	
	\subsection{Main objectives of the paper}
	In this paper, we aim to address two key questions pertaining to the double nonlinear porous medium equation \eqref{main_eqn_p>2}, where $f(u)$ satisfies conditions $(H)$:
	
	\textbf{Question 1:} Can finite time blow-up results for initial-boundary value problem \eqref{main_eqn_p>2} be obtained when
	\begin{align*}
	J(u_0) := \frac{1}{p} \int_{\Omega} |\nabla u^m_0(x)|^p dx - \int_{\Omega} (F(u_0(x))-\sigma) dx<0
	\end{align*} and $m>1$?
	
	For the case of $m=1$, Chung and Choi have already demonstrated the finite-time blow-up of solutions in \cite{Chung-Choi}.

	\textbf{Question 2:} Is it possible to extend the potential wells method to the initial-boundary value problem \eqref{main_eqn_p>2} for the subcritical initial energy scenario $0<J(u_0)\leq d$? Here 
	\begin{equation*}
	d := \inf \{J(u):  u^m \in W_0^{1,p}(\Omega)\backslash \{ 0 \}, I(u)=0\},
	\end{equation*} 
	and 
	\begin{align*}
	J(u) & = \frac{1}{p} || \nabla u^m ||^p_{L^p(\Omega)} - \int_{\Omega} [F(u)-\sigma] dx,
	\end{align*}
	with $u^m \in W_0^{1,p}(\Omega)$, $2 \leq p<n$, and $m\geq 1$. We define the Nehari functional as
	\begin{equation*}
	I(u)  =||\nabla u^m||^p_{L^p(\Omega)} - \int_{\Omega } u^mf(u)dx.
	\end{equation*}

	To the best of the authors' knowledge, the potential wells method has not yet been employed to examine porous medium equations featuring a nonlinear source term.

	We prove several results on the initial-boundary value problem \eqref{main_eqn_p>2}, each dependent on the condition (H) imposed on $f(u)$ and the initial energy $J(u_0)$:
	\begin{itemize}
		\item If the condition (H) on $f(u)$ holds and the initial energy $J(u_0)<0$, then the positive solution blows up in a finite time (Theorem \ref{thm_p>2}). This extends the blow-up results for the semilinear heat equation obtained by Chung and Choi in \cite{Chung-Choi}.
		\item If the condition (H) on $f(u)$ holds and the initial energy $J(u_0)<d$, then the initial-boundary value problem \eqref{main_eqn_p>2} has the invariant sets of solutions (Theorem \ref{thm-inv-sets}).
		\item If the condition (H) on $f(u)$ holds, the initial energy $J(u_0)\leq d$ and $I(u_0)\geq 0$, then the initial-boundary value problem \eqref{main_eqn_p>2} has a global weak solution (Theorem \ref{thm-GWS} and Theorem \ref{thm-GWS-d}). Furthermore, when $pm > m+1$, the solution decreases at a polynomial rate, while for $pm = m+1$, the rate of decrease is exponential (Theorem \ref{thm-global} and Theorem \ref{thm-asym}).
		\item If the condition (H) on $f(u)$ holds, the initial energy $J(u_0)\leq d$ and $I(u_0)< 0$, then the positive solution of the initial-boundary value problem \eqref{main_eqn_p>2} blows up in a finite time (Theorem \ref{thm-blow} and Theorem \ref{thm-blow-d}).
	\end{itemize}
	
	This paper is structured as follows: in Section \ref{sec-2} we present the blow-up solution of the double nonlinear porous medium equation \eqref{main_eqn_p>2} for the negative initial energy. In Section \ref{sec-3} we introduce a family of potential wells for the double nonlinear porous medium equation \eqref{main_eqn_p>2} and discuss their properties. The main results are proven in Section \ref{sec-4} and \ref{sec-5} for initial energy $J(u_0)\leq d$ and $I(u_0)\geq$ or $I(u_0)<0$. We conclude with a final discussion and present some open questions related to the problem \eqref{main_eqn_p>2}.
	
	\begin{defn}
		We say a function
		\begin{equation*}
		u^m \in L^{\infty} (0,T; W_0^{1,p}(\Omega)), \, \text{  with } (u^{\frac{m+1}{2}})_t \in L^{2} (0,T; L^{2}(\Omega)),
		\end{equation*}
		is a weak solution of the initial-boundary value problem \eqref{main_eqn_p>2} provided
		\begin{equation}\label{eq-weak-solution}
		( u_t, v)  +  (|\nabla u^m|^{p-2}\nabla u^m, \nabla v ) =  (f(u),v)
		\end{equation}
		for each $v \in W^{1,p}_0(\Omega)$ and a.e. time $0\leq t \leq T$ with
		\begin{equation*}
		u^m(x,0) = u_0^m(x) \,\,\, \text{ in } \,\, W_0^{1,p}(\Omega).
		\end{equation*}
	\end{defn}
	\section{Blow-up results for $J(u_0)<0$}\label{sec-2}
	In this section, we discuss the blow-up of the positive solution to the double nonlinear porous medium equation \eqref{main_eqn_p>2} in the case $J(u_0)<0$.

	Below we give a main result of this section.
	\begin{thm}\label{thm_p>2}
		Let $\Omega$ be a bounded domain of $\mathbb{R}^n$ with smooth boundary $\partial \Omega$. Let a function $f$ satisfy
		\begin{equation}\label{condt_p}
		\alpha F(u) \leq u^{m} f(u) + \beta u^{pm} +\alpha\sigma,\,\,\, u>0,
		\end{equation}
		where $m\geq 1$ and $\sigma >0$,
		\begin{equation}
		F(u)=m\int_{0}^{u}s^{m-1}f(s)ds,
		\end{equation}
		and $0<\beta\leq \frac{\lambda_{1,p}(\alpha - p)}{p}$. If $u_0^m \in L^{\infty}(\Omega)\cap W_0^{1,p}(\Omega)$ satisfies the inequality
		\begin{equation}\label{condt_p0}
		J(u_0)=\frac{1}{p}\int_{\Omega} |\nabla u_0^m(x)|^p dx - \int_{\Omega} \left( F(u_0(x)) -\sigma\right) dx <0,
		\end{equation}
		then there cannot exist a positive solution $u$ of \eqref{main_eqn_p>2} existing for all times $T^*$ and a positive constant $M_0$ such that
		\begin{equation}
		0<T^*\leq \frac{M_0}{\sigma \int_{\Omega}u_0^{m+1}(x)dx},
		\end{equation}
		in the sense that
		\begin{equation}
		\lim_{t\rightarrow T^*} \int_{0}^t \int_{\Omega} u^{m+1} (x,\tau) dx d\tau = +\infty.
		\end{equation}
	\end{thm}
	
	We introduce a lemma that will be helpful in our analysis.
	\begin{lem}[Theorem 1.1, \cite{KL06}]\label{lem1}
		There exists $\lambda_{1,p}>0$ and $0<w \in W^{1,p}_0(\Omega)$ in $\Omega \subset \mathbb{R}^n$ such that
		\begin{align}
		\begin{cases}
		\nabla \cdot ( |\nabla w(x)|^{p-2}\nabla w(x)) + \lambda_{1,p} w(x)=0, \,\,\,& x \in \Omega, \\
		w(x)  =0,  \,\,\, &x \in \partial \Omega,
		\end{cases}
		\end{align}
		where $\lambda_{1,p}$ is given by
		\begin{equation*}
		\lambda_{1,p} := \inf_{v\in W_0^{1,p}(\Omega)} \frac{\int_{\Omega}|\nabla v|^pdx}{\int_{\Omega}|v|^pdx}>0.
		\end{equation*}
		Recall that $\lambda_{1,p}$ is the first eigenvalue of $p$-Laplace operator and $w$ is a corresponding eigenfunction.
	\end{lem}
	Denote
	\begin{align*}
	J(u) := \frac{1}{p} \int_{\Omega} |\nabla u^m(x,t)|^p dx - \int_{\Omega} (F(u(x,t))-\sigma) dx,
	\end{align*}
	and
	\begin{align*}
	J(u_0) := \frac{1}{p} \int_{\Omega} |\nabla u^m_0(x)|^p dx - \int_{\Omega} (F(u_0(x))-\sigma) dx.
	\end{align*}
	We know that
	\begin{equation}\label{Fp}
	J(u) = J(u_0) + \int_{0}^t  J'(u)d\tau,
	\end{equation}
	where
	\begin{align*}
	\int_{0}^t J'(u)d\tau  &=   \frac{1}{p} \int_{0}^t\int_{\Omega} \frac{d}{d\tau}|\nabla u^m(x,\tau)|^p dx d\tau  - \int_{0}^t \int_{\Omega} \frac{d}{d\tau} (F(u(x,\tau))+\sigma) dx d\tau\\
	& =    \int_{0}^t\int_{\Omega} |\nabla u^m(x,\tau)|^{p-2} \nabla u^m(x, \tau) \cdot \nabla (u^m(x,\tau))_{\tau} dx d\tau\\
	& -  \int_{0}^t \int_{\Omega} F_u(u(x,\tau)) u_{\tau}(x,\tau) dx d\tau\\
	%& = \frac{2}{\mu} \int_{0}^t\int_{\Omega} \Delta u^m(x,\tau) (u^{m}(x,\tau))_{\tau} dx d\tau  + \frac{2}{\mu}\int_{0}^t \int_{\Omega} f(u(x,\tau)) (u^{m}(x,\tau))_{\tau} dx d\tau \\
	& = - \int_{0}^t \int_{\Omega} [\Delta_p u^m(x, \tau) + f(u)](u^m(x,\tau))_{\tau} dx d\tau \\
	& = -m\int_{0}^t \int_{\Omega} u^{m-1}(x,\tau) u_{\tau}^2(x,\tau) dx d\tau.
	\end{align*}
	This yields
	\begin{equation*}
	J(u) + m\int_{0}^t \int_{\Omega} u^{m-1}(x,\tau) u_{\tau}^2(x,\tau) dx d\tau= J(u_0).
	\end{equation*}

	\begin{proof}[Proof of Theorem \ref{thm_p>2}]
		Let
		\begin{equation}\label{I_p}
		E_p(t) := \int_{0}^t \int_{\Omega} u^{m+1}(x, \tau) dx d\tau + M_0, \,\, t\geq 0,
		\end{equation}
		with the positive constant $M_0$. Then we have
		\begin{equation*}
		E'_p(t)=\int_{\Omega} u^{m+1}(x,t) dx = (m+1)\int_{\Omega} \int_{0}^t u^{m}(x,\tau) u_{\tau}(x,\tau) d\tau dx + \int_{\Omega} u^{m+1}_0(x) dx.
		\end{equation*}
		By using the condition \eqref{condt_p}, Lemma \ref{lem1} and $0<\beta\leq \frac{\lambda_{1,p}(\alpha - p)}{p}$, we estimate the second derivative of $E_p(t)$ with respect to time 	
		\begin{align*}
		E''_p(t) &=(m+1) \int_{\Omega} u^{m}(x,t) u_t(x,t) dx\\
		& = (m+1) \int_{\Omega} u^{m}(x,t)  \Delta_p u^m(x,t) + (m+1)\int_{\Omega} u^{m}(x,t)f(u(x, t))dx \\
		& = -(m+1) \int_{\Omega} |\nabla u^m(x,t)|^p dx + (m+1)\int_{\Omega} u^{m}(x,t)f(u(x,t))dx
		\\
		& \geq - (m+1) \int_{\Omega} |\nabla u^m(x,t)|^p dx + (m+1) \int_{\Omega} \left[ \alpha F(u(x,t)) -\beta u^{pm} (x,t) -\alpha \sigma \right] dx\\
		& =-\alpha (m+1) \left[ \frac{1}{p} \int_{\Omega} |\nabla u^m(x,t)|^p dx - \int_{\Omega} (F(u(x,t))-\sigma) dx \right]  \\
		&+ \left(\frac{\alpha(m+1)}{p}-m-1\right) \int_{\Omega} |\nabla u^m(x,t)|^p dx -\beta (m+1) \int_{\Omega} u^{pm}(x, t) dx\\
		& \geq -\alpha(m+1) J(u)  +  (m+1)\left[ \frac{\lambda_{1,p}(\alpha-p)}{p}\ - \beta \right]\int_{\Omega} u^{pm}(x, t) dx \\
		& \geq -\alpha(m+1) J(u).
		\end{align*}
		
		Therefore, $E''_p(t)$ can be rewritten in the following form
		\begin{align*}
		E''_p(t) \geq -\alpha (m+1) J(u_0) + \alpha m(m+1) \int_{0}^t \int_{\Omega} u^{m-1}(x,\tau) u_{\tau}^2(x,\tau) dx d\tau,
		\end{align*}
		and
		\begin{align*}
		(	E'_p(t) )^2 &\leq (m+1)^2(1+\delta) \left(\int_{0}^t \int_{\Omega} u^{m+1}(x, \tau) dx d\tau\right)\left( \int_{0}^t \int_{\Omega} u^{m-1}(x,\tau) u_{\tau}^2(x,\tau)dxd\tau \right)   \\
		&+ \left( 1+ \frac{1}{\delta}\right)\left( \int_{\Omega} u_0^{m+1}(x)dx \right)^2.
		\end{align*}
		Then by taking $\varepsilon=\delta= \frac{\sqrt{\alpha m(m+1)}}{m+1}-1>0$, we discover
		\begin{align*}
		&E''_p(t) E_p (t) - (1+\varepsilon) (E'_p(t))^2\\
		&\geq  -\alpha M_0(m+1) J(u_0)
		\\&+ \alpha m(m+1) \left( \int_{0}^t\int_{\Omega} u^{m+1}(x, \tau) dx d\tau \right) \left(\int_{0}^t \int_{\Omega} u_{\tau}^2(x,\tau) u^{m-1}(x,\tau) dx d\tau  \right)\\
		&- (m+1)^2(1+\varepsilon) (1+\delta) \left(\int_{0}^t \int_{\Omega} u^{m+1}(x, \tau) dx d\tau\right)\left( \int_{0}^t \int_{\Omega} u^{m-1} (x, \tau)u_{\tau}^2(x,\tau)dxd\tau \right)\\
		& - (1+\varepsilon)\left( 1+ \frac{1}{\delta}\right)\left( \int_{\Omega} u_0^{m+1}(x)dx \right)^2 \\
		&\geq - \alpha M_0(m+1) J(u_0) - (1+\varepsilon)\left( 1+ \frac{1}{\delta}\right)\left( \int_{\Omega} u_0^{m+1}(x)dx \right)^2.
		\end{align*}	
		By assumption $J(u_0)<0$, thus if we select $M$ sufficiently large we have
		\begin{equation}
		E''_p(t) E_p (t) - (1+\varepsilon) (E'_p(t))^2 > 0.
		\end{equation}
		We can see that the above expression for $t\geq 0$ implies
		\begin{equation*}
		\frac{d}{dt} \left[ \frac{E'_p(t)}{E^{\varepsilon+1}_p(t)} \right] >0,
		\end{equation*} hence
		\begin{equation*}
		\begin{cases}
		E'_p \geq \left[ \frac{E'_p(0)}{E^{\varepsilon+1}_p(0)} \right] E^{1+\varepsilon}_p(t),\\
		E_p(0)=M_0.
		\end{cases}
		\end{equation*}
		Then for $\varepsilon = \frac{\sqrt{\alpha m(m+1)}}{m+1}-1>0$, we arrive at
		\begin{equation*}
		E_p(t) \geq \left( \frac{1}{M^{\sigma}_0}-\frac{ \varepsilon \int_{\Omega} u^{m+1}_0(x)dx }{M^{\varepsilon+1}_0} t\right)^{-\frac{1}{\varepsilon}}.
		\end{equation*}
		Then the blow-up time $T^*$ satisfies
		\begin{equation*}
		0<T^*\leq \frac{M_0}{\varepsilon \int_{\Omega} u_0^{m+1}dx}.
		\end{equation*}
		That completes the proof.
	\end{proof}	
	\section{Family of Potential Wells}\label{sec-3}
	Define
	\begin{align*}
	J(u) & = \frac{1}{p} || \nabla u^m ||^p_{L^p(\Omega)} - \int_{\Omega} [F(u)-\sigma] dx,
	\end{align*}
	where  $u^m \in W_0^{1,p}(\Omega)$, $2\leq p<n$, and $m\geq 1$. We denote the Nehari functional by
	\begin{equation*}
	I(u)  =||\nabla u^m||^p_{L^p(\Omega)} - \int_{\Omega } u^mf(u)dx.
	\end{equation*}
	
	The depth of potential well is defined by
	\begin{equation*}
	d := \inf \{J(u):  u^m \in W_0^{1,p}(\Omega)\backslash \{ 0 \}, I(u)=0\}.
	\end{equation*}
	Also the quantity $d$ can be understood as a mountain pass energy. For instance, there is the valley. If we pour water into the valley containing a local minimum at the origin, as the water level rises in the direction of the potential energy and passes at a local maximum $d$.
	
	\begin{lem}\label{lem-1}
		Let $f(u)$ satisfy $(H)$, then we have
		\begin{itemize}
			\item[$(a)$] $wf'(w)- m(p-1) f(w) \geq 0$ holds for $w \geq 0,$ $w \in W_0^{1,p}(\Omega)$;
			\item[$(b)$] $|F(u)-\sigma|\leq A |u|^{\gamma}$ for some $A>0$ and all $u\in R$;
			\item[$(c)$] $F(u)-\sigma \geq B |u|^{\lambda}$ for $|u|\geq 1$ with $B = \frac{F(\widetilde{u})}{\widetilde{u}^{\lambda}}>0$ and $\lambda = \frac{\alpha}{1+\beta}>pm$.
		\end{itemize}
	\end{lem}
	\begin{proof}[Proof of Lemma \ref{lem-1}]
		$(a)$ is proved by the variational method.   We follow the proof of Lemma 2.1 in Payne-Sattinger's paper \cite{PaSa}. It is known that a critical point of functional 
		\begin{equation*}
		J(u) = \frac{1}{p} \int_{\Omega} |\nabla u^m |^p dx - \int_{\Omega}F(u) dx,
		\end{equation*}
		is a solution of Euler-Lagrange equation ($p$-Poisson equation)
		\begin{equation*}
		-\Delta_p w^m = f(w),
		\end{equation*}
		where the existence and uniqueness of the solution is presented in Chapter 2 of \cite{Lions}. Given the functional $J(u)$ exhibits concavity as a consequence of the convex nature of $f(u)$ for $u>0$, we can analyze its behavior at non-trivial critical points. In particular, at a non-trivial critical point $w$, the second variation of $J(u)$, denoted as $i''(0)$, will not be positive definite. Using the fact $i''(0)\leq 0$, we intend to demonstrate that $wf'(w) -m(p-1)f(w)\geq 0$. Define
		\begin{equation*}
		i(\tau) = J(w+\tau v) = \frac{1}{p} \int_{\Omega } |\nabla (w+\tau v)^m|^p dx - \int_{\Omega } F(w+\tau v)dx.
		\end{equation*}
		Let us compute the second variation of $i(\tau)$. The first derivative yields
		\begin{align*}
		i'(\tau) &= m\int_{\Omega }  |\nabla (w+\tau v)^m|^{p-2} \nabla (w+\tau v)^m \cdot \nabla [(w+\tau v)^{m-1}v ]dx\\& - m\int_{\Omega } (w+\tau v)^{m-1}f(w+\tau v) v dx,
		\end{align*}
		and
		\begin{align*}
		i''(\tau) =& m^2(p-2)\int_{\Omega } |\nabla (w+\tau v)^m|^{p-4} |\nabla (w+\tau v)^m \cdot \nabla [(w+\tau v)^{m-1}v]|^2 dx\\
		&+ m^2\int_{\Omega }|\nabla (w+\tau v)^m|^{p-2} |\nabla [(w+\tau v)^{m-1}v] |^2 dx\\
		& + m(m-1)\int_{\Omega } |\nabla (w+\tau v)^m|^{p-2}  \nabla(w+\tau v)^m\cdot \nabla [(w+\tau v)^{m-2}v^2] dx\\
		& - m(m-1)\int_{\Omega } (w+\tau v)^{m-2} v^2 f(w+\tau v) dx - m\int_{\Omega } (w+\tau v)^{m-1} v^2f'(w+\tau v) dx.
		\end{align*}
		Then the second variation of $J$ at a critical point $w$ is
		\begin{align*}
		i''(0)  =& m^2(p-2)\int_{\Omega } |\nabla w^m|^{p-4} |\nabla w^m \cdot \nabla [w^{m-1}v]|^2 dx+ m^2 \int_{\Omega }|\nabla w^m|^{p-2} |\nabla [w^{m-1}v] |^2 dx\\
		& + m(m-1)\int_{\Omega } |\nabla w^m|^{p-2}  \nabla w^m\cdot \nabla [w^{m-2}v^2] dx\\
		& - m(m-1)\int_{\Omega } w^{m-2} v^2 f(w) dx - m\int_{\Omega }  w^{m-1} v^2f'(w) dx.
		\end{align*}
		Let us choose the test function $v$ to be $w$, then we compute
		\begin{align*}
		i''(0)  =& m(mp-1) \int_{\Omega } |\nabla w^m|^pdx - m(m-1)\int_{\Omega } w^{m} f(w) dx - m\int_{\Omega }  w^{m+1}f'(w) dx\\
		=& m\int_{\Omega } w^m \left[ (mp-1) (- \Delta_p w^m)  - (m-1) f(w) -   wf'(w)\right] dx \\
		=& - m\int_{\Omega } w^m \left[    wf'(w)- m(p-1) f(w)\right] dx.
		\end{align*}
		Since $i''(0)\leq 0$, we have $wf'(w)- m(p-1) f(w) \geq 0$.
		
		$(b)$ From $(H)$, the growth condition $|u^mf(u)|\leq \gamma |F(u)-\sigma|$ implies
		\begin{align*}
		\frac{|u^{m-1}f(u)|}{|F(u)-\sigma|} \leq \frac{\gamma}{|u|} \,\,\, \text{ for } \,\,\, u \neq 0.
		\end{align*}
		Taking any known fixed $\widetilde{u} \neq 0$, we integrate above inequality from $\widetilde{u}$ to $u$, which gives
		\begin{align*}
		\int_{\widetilde{u}}^u \frac{|s^{m-1}f(s)|}{|F(s)-\sigma|} ds &\leq 	\int_{\widetilde{u}}^u \frac{\gamma}{|s|}ds,
		\end{align*}
		and by using $F'(u)=u^{m-1}f(u)$, we have
		\begin{align*}
		\ln \frac{|F(u)-\sigma|}{|F(\widetilde{u})-\sigma|}\leq \ln \frac{|u|^{\gamma}}{|\widetilde{u}|^{\gamma}},
		\end{align*}
		that is $|F(u)-\sigma| \leq A |u|^{\gamma},$ where $A = \frac{|F(\widetilde{u})-\sigma|}{|\widetilde{u}|^{\gamma}}$.
		
		$(c)$ From $(H)$ we have $\alpha F(u) \leq u^mf(u) + \beta u^{pm} +\alpha  \sigma$. For $|u|\geq 1$, we write
		\begin{align*}
		\frac{\alpha}{u}  \leq  \frac{u^{m-1}f(u)}{F(u)-\sigma} +  \frac{\beta u^{pm-1}}{F(u)-\sigma}  \leq (1+\beta)  \frac{u^{m-1}f(u)}{F(u)-\sigma},
		\end{align*}
		then we get
		\begin{equation}
		\frac{u^{m-1}f(u)}{F(u)-\sigma} \geq \frac{\lambda}{u} ,
		\end{equation}
		where $\lambda= \frac{\alpha}{1+\beta}$. As the previous approach on $(b)$, we get $F(u)-\sigma \geq B |u|^{\lambda} $ where $B = \frac{F(\widetilde{u})}{\widetilde{u}^{\lambda}}$.
	\end{proof}

	\begin{lem}\label{lem-2}
		Let $f(u)$ satisfy $(H)$. Let $u^m\in W_0^{1,p}(\Omega)$, $||\nabla u^m ||_{L^p(\Omega)}\neq 0$, and
		\begin{align}
		\phi(\epsilon) = \frac{1}{\epsilon^{m(p-1)}} \int_{\Omega } u^{m} f(\epsilon u) dx.
		\end{align}
		Then we have the following properties of function $\phi$:
		\begin{itemize}
			\item $\phi(\epsilon)$ is increasing on $0< \epsilon < \infty$.
			\item $\lim_{\epsilon \rightarrow 0}\phi(\epsilon)=0$ and $\lim_{\epsilon \rightarrow +\infty}\phi(\epsilon)=+\infty$.
		\end{itemize}
	\end{lem}
	\begin{proof}[Proof of Lemma \ref{lem-2}]
		Let
		\begin{align*}
		\frac{d\phi(\epsilon)}{d\epsilon} &= \frac{1}{\epsilon^{m(p-1)+1}} \int_{\Omega } [\epsilon u^{m+1}f'(\epsilon u) -  m(p-1)u^{m}f(\epsilon u)]  dx \\
		&= \frac{1}{\epsilon^{m(p-1)+1}} \int_{\Omega }  u^{m}[\epsilon u f'(\epsilon u) - m(p-1)f(\epsilon u)]  dx >0.
		\end{align*}
		By $(a)$ in Lemma \ref{lem-1}, we see that $\phi(\epsilon)$ is increasing on $0<\epsilon < \infty$.
		
		The property $\lim_{\epsilon \rightarrow 0}\phi(\epsilon)=0$ comes from the fact that
		\begin{align*}
		|\phi (\epsilon)|& \leq \frac{1}{\epsilon^{pm}} \int_{\Omega } |\epsilon^m u^{m} f(\epsilon u)|dx \leq \frac{\gamma A}{\epsilon^{pm}} \int_{\Omega } |\epsilon u|^{\gamma}dx\\& = \gamma \epsilon^{\gamma -pm} A || u ||^{\gamma}_{L^{\gamma}(\Omega)}.
		\end{align*}
		
		The property $\lim_{\epsilon \rightarrow +\infty}\phi(\epsilon)=+\infty$ can be shown by using the fact $u^mf(u)\geq 0$ and  $(c)$ in Lemma \ref{lem-1} as follows
		\begin{align*}
		\phi(\epsilon) &= \frac{1}{\epsilon^{pm}}\int_{\Omega } \epsilon^m u^m f(\epsilon u) dx \\&\geq \frac{1}{\epsilon^{pm}}\int_{\Omega_{\epsilon} } \epsilon^m u^m f(\epsilon u) dx\\
		& \geq \frac{1}{\epsilon^{pm}}\int_{\Omega_{\epsilon} } [\alpha F(\epsilon u) - \beta \epsilon^{pm} u^{pm} -  \alpha \sigma  ]dx \\
		& \geq \alpha B \epsilon^{\lambda -pm} \int_{\Omega_{\epsilon} } |u|^{\lambda} dx - \beta \int_{\Omega_{\epsilon} } u^{pm} dx,
		\end{align*}
		where $\Omega_{\epsilon} = \{ x \in \Omega \,\, | \,\, u(x)\geq \frac{1}{\epsilon} \}$ and
		\begin{equation*}
		\lim_{\epsilon \rightarrow + \infty} \int_{\Omega_{\epsilon} } |u|^{\lambda} dx = || u ||^{\lambda}_{L^{\lambda}(\Omega)} >0.
		\end{equation*}
		This completes the proof.
	\end{proof}
	When nonlinearity term $f(u)$ satisfies condition $(H)$, then all non-trivial critical points are a priori unstable equilibria for porous equation \eqref{main_eqn_p>2}. The next lemma will provide the local minimum of $J$, a critical point $\epsilon^*$ and positiveness of a depth of potential well.
	\begin{lem}\label{lem-3} 	Let $f(u)$ satisfy $(H)$, $u^m\in W_0^{1,p}(\Omega)$, and $||\nabla u^m ||_{L^p(\Omega)} \neq 0$. Then
		\begin{itemize}
			\item[(i)] $ \lim_{\epsilon \rightarrow 0}J(\epsilon u)=0$.
			\item[(ii)] $ \frac{d}{d\epsilon} J(\epsilon u) = \frac{1}{\epsilon} I(\epsilon u)$.
			\item[(iii)] $ \lim_{\epsilon \rightarrow + \infty}J(\epsilon u)= -\infty$.
			\item[(iv)] There exists a unique $\epsilon^*=\epsilon^*(u)>0$ such that
			\begin{equation}
			\frac{d}{d\epsilon} J(\epsilon u)|_{\epsilon=\epsilon^*} = 0.
			\end{equation}
			\item[(v)] $J(\epsilon u)$ is increasing on $0 \leq \epsilon \leq \epsilon^*$, decreasing on $\epsilon^* \leq \epsilon < \infty$ and takes the maximum at $\epsilon = \epsilon^*$.
			\item[(vi)] $I(\epsilon u)>0$ for $0<\epsilon < \epsilon^*$, $I(\epsilon u)<0$ for $\epsilon^*<\epsilon < \infty$  and $I(\epsilon^*u)=0$.
		\end{itemize}
	\end{lem}
	\begin{proof}[Proof of Lemma \ref{lem-3}]
		The proof of (i) follows from $|F(u)-\sigma|\leq A |u|^{\gamma}$ and
		\begin{equation*}
		J(\epsilon u) = \frac{\epsilon^{pm}}{p} \int_{\Omega } |\nabla u^m|^p dx - \int_{\Omega } [F(\epsilon u) -\sigma ]dx.
		\end{equation*}
		The proof of (ii) follows from
		\begin{align*}
		\frac{d}{d\epsilon} J(\epsilon u)& = m\epsilon^{pm-1} \int_{\Omega } |\nabla u^m|^p dx - m\epsilon^{m-1}\int_{\Omega } u^mf(\epsilon u)dx\\& = \frac{m}{\epsilon}I(\epsilon u).
		\end{align*}
		Let
		\begin{align*}
		J(\epsilon u) &= \frac{\epsilon^{pm}}{p} \int_{\Omega } |\nabla u^m|^p dx - \int_{\Omega } [F(\epsilon u) -\sigma] dx\\
		&\leq  \frac{\epsilon^{pm}}{p} \int_{\Omega } |\nabla u^m|^p dx  - B \epsilon^{\lambda} \int_{\Omega_{\epsilon} } |u|^{\lambda}dx.
		\end{align*}
		Since $\lambda >pm$, $J(\epsilon u)$ goes to $-\infty$ for $\epsilon \rightarrow +\infty$. This guaranties the existence of a critical point $\epsilon^*$.
		
		The case (iv), the uniqueness of critical point, can be proved by supposing that there two roots of equation $\frac{dJ(\epsilon u)}{d\epsilon}=0$ as $\epsilon_1<\epsilon_2$. Then we have
		\begin{equation*}
		\epsilon_1^{pm-m} \int_{\Omega }|\nabla u^m|^pdx - \int_{\Omega } u^mf(\epsilon_1 u)dx =0,
		\end{equation*}
		and
		\begin{equation*}
		\epsilon_2^{pm-m} \int_{\Omega }|\nabla u^m|^pdx  - \int_{\Omega } u^mf(\epsilon_2 u)dx =0.
		\end{equation*}
		Elimination of $\int_{\Omega }|\nabla u^m|^p  $ from above equations gives
		\begin{equation*}
		\int_{\Omega } u^m \left[\frac{f(\epsilon_2 u)}{\epsilon_2^{pm-m}}- \frac{f(\epsilon_1 u)}{\epsilon_1^{pm-m}} \right]dx = 0.
		\end{equation*}
		This expression can be rewritten by the following substitution $w=\epsilon_1 u$ and $\epsilon = \epsilon_2/\epsilon_1>1$ as
		%\begin{align*}
		%&\int_{\Omega }\frac{ u^m f(\epsilon w)}{\epsilon_2^{pm-m}}dx- \int_{\Omega } \frac{w^m f(w) }{\epsilon_1^{pm}}dx =0 \\
		%&\int_{\Omega }\frac{ w^m f(\epsilon w)}{\epsilon^{pm-m}}dx - \int_{\Omega }w^m f(w)dx =0.
		%\end{align*}
		\begin{equation}\label{eq-2.4}
		\frac{ 1}{\epsilon^{pm-m}}\int_{\Omega } w^m f(\epsilon w)dx = \int_{\Omega }w^m f(w)dx.
		\end{equation}
		It is easy to see that equation \eqref{eq-2.4} does not hold.
		
		The case (v) and (vi) follow from Lemma \ref{lem-2} and
		\begin{equation*}
		\frac{d}{d\epsilon} J(\epsilon u) = \frac{m}{\epsilon} I(\epsilon u)=m\epsilon^{pm-1} \left[\int_{\Omega } |\nabla u^m|^p dx -\phi(\epsilon) \right].
		\end{equation*}
		The proof is completed.
	\end{proof}

	In order to introduce the family of potential wells $\{W_{\delta}\}$ and  $\{V_{\delta}\}$, we define
	\begin{align*}
	I_{\delta}(u) = \delta \int_{\Omega } |\nabla u^m|^pdx - \int_{\Omega } u^mf(u)dx, \,\,\, \delta >0,
	\end{align*}
	and
	\begin{equation*}
	d(\delta) := \inf \{J(u): u^m \in W_0^{1,p}(\Omega)\backslash \{ 0 \}, I_{\delta}(u)=0\}.
	\end{equation*}

	\begin{lem}\label{lem-4}
		Let $f(u)$ satisfy $(H)$. Suppose that $u^m \in W_0^{1,p}(\Omega)$ and
		\begin{align*}
		r(\delta) = \left( \frac{\delta}{a C^{\gamma}_*} \right)^{\frac{1}{\gamma-p}} \,\,\, \text{and} \,\,\, a = \sup \frac{u^m f(u)}{|u^m|^{\gamma}},
		\end{align*}
		where $C_*$ is the embedding constant form $W_0^{1,p}(\Omega)$ into $L^{\gamma}(\Omega)$. Then we have
		\begin{itemize}
			\item[(i)] If $0<|| \nabla u^m ||_{L^p(\Omega)} < r(\delta)$, then $I_{\delta}(u)>0$. In the case $\delta =1$, if $0< || \nabla u^m ||_{L^p(\Omega)}  <r(1)$, then $I(u)>0$.
			\item[(ii)] If $I_{\delta}(u)<0$, then $|| \nabla u^m ||_{L^p(\Omega)} >r(\delta)$. In the case $\delta =1$, if $I(u)<0$, then $|| \nabla u^m ||_{L^p(\Omega)} >r(1)$.
			\item[(iii)] If $I_{\delta}(u)=0$, then $|| \nabla u^m ||_{L^p(\Omega)} \geq r(\delta)$. In the case $\delta =1$, if $I(u)=0$, then $|| \nabla u^m ||_{L^p(\Omega)} \geq r(1)$.
		\end{itemize}
	\end{lem}
	\begin{lem}[Sobolev inequality]
		Let $\Omega$ be a bounded, open subset of $\mathbb{R}^n$, and suppose $\partial \Omega$ is $C^1$. Assume $1< p<n$, $p^* = \frac{np}{n-p}$, and $f \in W^{1,p}(\Omega)$. Then $f \in L^{p^{*}}(\Omega)$, with estimate
		\begin{equation*}
		|| f ||_{L^{p^*}(\Omega)} \leq C || f ||_{W^{1,p}(\Omega)},
		\end{equation*}
		the constant $C$ depending only on $p$, $n$, and $\Omega$.
	\end{lem}
	\begin{proof}[Proof of Lemma \ref{lem-4}]
		
		(i)	From $0<|| \nabla u^m ||_{L^p(\Omega)}<r{(\delta)}$ we have
		\begin{align*}
		\int_{\Omega } u^m f(u) dx &\leq \int_{\Omega } |u^m f(u)| dx \leq a \int_{\Omega } |u^m|^{\gamma} dx = a || u^m||^{\gamma}_{L^{\gamma}(\Omega)}\\
		&\leq a C_*^{\gamma} || \nabla u^m||^{\gamma-p}_{L^p(\Omega)}|| \nabla u^m||^{p}_{L^p(\Omega)}< \delta || \nabla u^m||^p_{L^p(\Omega)},
		\end{align*}
		which gives $I_{\delta}(u)>0$.
		
		(ii) From $I_{\delta}(u)<0$ we get
		\begin{align*}
		\delta || \nabla u^m ||^p_{L^p(\Omega)}& < \int_{\Omega } u^m f(u) dx \leq a || u^m ||^{\gamma}_{L^{\gamma}(\Omega)} \\&\leq aC_*^{\gamma} || \nabla u^m||^{\gamma -p}_{L^p(\Omega)}  || \nabla u^m||^{p}_{L^p(\Omega)},
		\end{align*}
		which implies $||\nabla u^m ||_{L^p(\Omega)} >r(\delta)$.
		
		(iii)  From $I_{\delta}(u)=0$ we obtain
		\begin{align*}
		\delta || \nabla u^m ||^p_{L^p(\Omega)} &\leq \int_{\Omega } u^m f(u) dx \leq a || u^m ||^{\gamma}_{L^{\gamma}(\Omega)}\\&\leq aC_*^{\gamma} || \nabla u^m||^{\gamma -p}_{L^p(\Omega)}  || \nabla u^m||^{p}_{L^p(\Omega)},
		\end{align*}
		which implies $||\nabla u^m ||_{L^p(\Omega)} \geq r(\delta)$.
	\end{proof}
	Next lemma describes the properties of $d(\delta)$ for a family of potential wells.
	\begin{lem}\label{lem-5}
		Let $f(u)$ satisfy $(H)$. Then
		\begin{itemize}
			\item[(i)] $d(\delta)>a(\delta)r^p(\delta)$ for $0<\delta< \frac{\alpha}{p} -\frac{\beta}{\lambda_{1,p}}$, where $a(\delta)= \frac{1}{p} - \frac{\delta}{\alpha} - \frac{\beta}{\lambda_{1,p}\alpha}$.
			\item[(ii)] $\lim_{\delta \rightarrow 0}d(\delta)=0$ and there exists a unique $b$, $\frac{\alpha}{p} - \frac{\beta}{\lambda_{1,p}} \leq b \leq \frac{\gamma }{p}$ such that $d(b)=0$ and $d(\delta)>0$ for $0<\delta<b$.
			\item[(iii)] $d(\delta)$ is strictly increasing on $0\leq \delta \leq 1$, strictly decreasing on $1\leq \delta \leq \infty$ and takes the maximum $d=d(1)$ at $\delta =1$.
		\end{itemize}
	\end{lem}
	\begin{proof}[Proof of Lemma \ref{lem-5}]
		
		(i) If $I_{\delta}(u) =0$ and $|| \nabla u||_{L^p(\Omega)}\neq 0$, then we have $|| \nabla u||_{L^p(\Omega)} \geq r(\delta)$ from Lemma \ref{lem-4}. By applying $I_{\delta}(u)=0$, Lemma \ref{lem1} and \ref{lem-4}, we have
		\begin{align*}
		J(u) &= \frac{1}{p} || \nabla u^m ||^p_{L^p(\Omega)} - \int_{\Omega } [F(u)-\sigma] dx\\
		&\geq  \frac{1}{p} || \nabla u^m ||^p_{L^p(\Omega)} - \frac{1}{\alpha} \int_{\Omega }u^mf(u)dx- \frac{\beta}{\alpha} || u^m||^p_{L^p(\Omega)} \\
		& \geq \left[\frac{1}{p}  - \frac{\delta}{\alpha} - \frac{\beta}{\lambda_{1,p}\alpha} \right]|| \nabla u^m ||^p_{L^p(\Omega)}\\
		& \geq a(\delta) r^p(\delta),
		\end{align*}
		where $0<\delta< \frac{\alpha}{pm} -\frac{\beta}{\lambda_{1,p}}$.
		
		$(ii)$ Lemma \ref{lem-2} implies that for $u^m \in W_0^{1,p}(\Omega)$, $||\nabla u^m||_{L^p(\Omega)}\neq 0$ and any $\delta >0$ we can define a unique $\epsilon(\delta) = \phi^{-1}(\delta || \nabla u^m||^p_{L^p(\Omega)})$ such that
		\begin{equation}\label{eq-3}
		\epsilon^{pm}\phi(\epsilon) = \int_{\Omega } (\epsilon u)^m f(\epsilon u) dx=\delta || \nabla (\epsilon u)^m||^p_{L^p(\Omega)} .
		\end{equation}
		Then for $I_{\delta}(\epsilon u)=0$ we get
		\begin{equation*}
		\lim_{\delta \rightarrow 0} \epsilon (\delta) =0 \,\, \text{ and } \,\, \lim_{\delta \rightarrow +\infty} \epsilon (\delta) = +\infty.
		\end{equation*}
		From Lemma \ref{lem-3} it follows that
		\begin{align*}
		\lim_{\delta \rightarrow 0} J(\epsilon u) = \lim_{\epsilon \rightarrow 0}J(\epsilon u)= 0  \,\, &\text{ and } \,\, \lim_{\delta \rightarrow 0} d(\delta) = 0, \\
		\lim_{\delta \rightarrow +\infty} J(\epsilon u) = \lim_{\epsilon \rightarrow +\infty}J(\epsilon u)= -\infty  \,\, &\text{ and } \,\, \lim_{\delta \rightarrow + \infty} d(\delta) = -\infty.
		\end{align*}
		Existence of $d(b)=0$ for $b\geq \alpha/p - \beta/\lambda_{1,p}$ and $d(\delta)>0$ for $0<\delta<b$ follows from the part $(i)$ and above expressions. Upper bound of $b\leq \gamma/p$ follows from the facts $|u^mf(u)|\leq \gamma |F(u)-\sigma|$ and $I_{\delta}(u)=0$, then
		\begin{align*}
		J(u) &= \frac{1}{p} || \nabla u^m||^p_{L^p(\Omega)}  - \int_{\Omega } [F(u)-\sigma]dx \\
		& \geq \frac{1}{p} || \nabla u^m||^p_{L^p(\Omega)} - \gamma^{-1} \int_{\Omega } u^mf(u)dx\\
		& = \left( \frac{1}{p} - \frac{\delta}{\gamma}\right) || \nabla u^m
		||^p_{L^p(\Omega)} < 0, \,\,\, \text{ for } \,\,\, \delta > \gamma/p.
		\end{align*}

		For the case $(iii)$, it is enough to show that $d(\delta')<d(\delta'')$ for any $0<\delta'<\delta''<1$ or $1<\delta''<\delta' <b$. In other word, we need to prove that for some $c(\delta',\delta'')>0$, we have
		$$J(v)>J(u)-c(\delta',\delta''),$$
		for any $u^m\in W_0^{1,p}(\Omega)$ and $v^m\in W_0^{1,p}(\Omega)$ with $I_{\delta''}(u)=0$, $|| \nabla ^mu||_{L^p(\Omega)} \neq 0$ and $I_{\delta'}(u)=0$, $|| \nabla v^m||_{L^p(\Omega)} \neq 0$, respectively.
		
		First, we define $\epsilon(\delta)$ for $u$ as \eqref{eq-3}, then $I_{\delta}(\epsilon(\delta)u)=0$ and $\epsilon(\delta'')=1$. Let us have $g(\epsilon)=J(\epsilon u)$, then
		\begin{align*}
		\frac{d}{d\epsilon} g(\epsilon)& = \frac{1}{\epsilon} \left( || \nabla (\epsilon u)^m ||^p_{L^p(\Omega)} - \int_{\Omega } \epsilon^m u^m f(\epsilon u) dx \right)\\& = (1-\delta )\epsilon^{pm-1} ||\nabla u^m ||^p_{L^p(\Omega)},
		\end{align*}
		where we use the fact $\delta || \nabla u^m ||^p_{L^p(\Omega)} = \int_{\Omega } u^mf(u)dx$.
		
		Now we choose $v=\epsilon(\delta')u$, then $I_{\delta'}(v)=0$ and $|| \nabla v^m ||_{L^p(\Omega)} \neq 0$. For $0<\delta'<\delta''<1$, we get
		\begin{align*}
		J(u) - J(v) &= J(\epsilon(\delta'') u) - J(\epsilon(\delta') u) \\
		&= g(1) - g(\epsilon(\delta')) \\
		&= (1-\delta')\epsilon(\delta')|| \nabla u^m||^p_{L^p(\Omega)}(1-\epsilon(\delta'))\\
		&> (1-\delta'')\epsilon(\delta')r^p(\delta'')(1-\epsilon(\delta')) \equiv c(\delta',\delta'').
		\end{align*}
		For $1<\delta''<\delta'<b$, we have
		\begin{align*}
		J(u) - J(v) &= J(\epsilon(\delta'') u) - J(\epsilon(\delta') u)\\
		&= g(1) - g(\epsilon(\delta')) \\
		&> (\delta'' -1)r^p(\delta'')\epsilon(\delta'') (\epsilon(\delta'')-1).
		\end{align*}
		This completes the proof.
	\end{proof}
	\begin{lem}\label{lem-6}
		Let $f(u)$ satisfy $(H)$ and $0<\delta<\frac{\alpha}{p}-\frac{\beta}{\lambda_{1,p}}$. Then the following properties hold
		\begin{itemize}
			\item[(i)] Suppose that $J(u)\leq d(\delta)$ and $I_{\delta}(u)>0$, then
			\begin{equation*}
			0<|| \nabla u^m||^p_{L^p(\Omega)} < \frac{d(\delta)}{a(\delta)}.
			\end{equation*}
			In the case $\delta=1$, if $J(u)\leq d$ and $I(u)>0$, then
			\begin{equation*}
			0<|| \nabla u^m||^p_{L^p(\Omega)} < \frac{d}{a(1)}.
			\end{equation*}
			\item[(ii)] Suppose that $J(u)\leq d(\delta)$ and $I_{\delta}(u)=0$, then
			\begin{equation*}
			0<|| \nabla u^m||^p_{L^p(\Omega)} < \frac{d(\delta) }{a(\delta)}.
			\end{equation*}
			In the case $\delta=1$, if $J(u)\leq d$ and $I(u)=0$, then
			\begin{equation*}
			0<|| \nabla u^m||^p_{L^p(\Omega)} < \frac{d}{a(1)}.
			\end{equation*}
			\item[(iii)] Suppose that $J(u)\leq d(\delta)$ and $|| \nabla u^m||^p_{L^p(\Omega)}>d(\delta)/a(\delta)$, then $I_{\delta}(u)<0$. \\
			In the case $\delta=1$, if $J(u)\leq d$ and $|| \nabla u^m||^p_{L^p(\Omega)}>{d}/{a(1)}$, then $I(u)<0$.
		\end{itemize}
	\end{lem}
	\begin{proof}[Proof of Lemma \ref{lem-6}] The case $(i)$ follows from
		\begin{align*}
		J(u) &= \frac{1}{p} || \nabla u^m ||^p_{L^p(\Omega)} - \int_{\Omega } [F(u)- \sigma] dx\\
		&\geq  \frac{1}{p} || \nabla u^m ||^p_{L^p(\Omega)} - \frac{1}{\alpha} \int_{\Omega }u^mf(u)dx- \frac{\beta}{\alpha} \int_{\Omega }u^{pm}dx  \\
		& > \left[\frac{1}{p}  - \frac{\delta}{\alpha} - \frac{\beta}{\lambda_{1,p}\alpha} \right]|| \nabla u^m ||^p_{L^p(\Omega)}  \\
		& = a(\delta) || \nabla u^m ||^p_{L^p(\Omega)} .
		\end{align*}
		The proof of cases $(ii)$ and $(iii)$ follows from the similar argument.
	\end{proof}
	
	Then a family of potential wells can be defined for $0<\delta<b$ as follows
	\begin{align*}
	W_{\delta} &:= \{ u^m \in W_0^{1,p}(\Omega) \,\,| \,\, I_{\delta}(u)>0, J(u)<d(\delta) \} \cup \{0\},\\
	\overline{W}_{\delta} & := W_{\delta}\cup \partial W_{\delta}=\{ u^m \in W_0^{1,p}(\Omega) \,\,| \,\, I_{\delta}(u)\geq 0, J(u) \leq d(\delta) \},
	\end{align*}
	In the case $\delta=1$, we have
	\begin{align*}
	W & := \{ u^m \in W_0^{1,p}(\Omega) \,\,| \,\, I(u)>0, J(u)<d \} \cup \{0\},\\
	\overline{W} & := \{ u^m \in W_0^{1,p}(\Omega) \,\,| \,\, I(u)\geq 0, J(u)\leq d \}.
	\end{align*}
	In addition, we define
	\begin{align*}
	V_{\delta}&:= \{ u^m \in W_0^{1,p}(\Omega) \,\,| \,\, I_\delta(u)<0, J(u)<d(\delta) \},\\
	V &:=\{ u^m \in W_0^{1,p}(\Omega) \,\,| \,\, I(u)<0, J(u)<d \}.
	\end{align*}
	
	From the definition of $W_{\delta}$, $V_{\delta}$ and Lemma \ref{lem-5}, we derive the following properties:
	\begin{lem}\label{lem-7}
		If $0<\delta'<\delta''\leq 1$ and $1\leq \delta''<\delta' <b $, then $W_{\delta'} \subset W_{\delta''}$ and $V_{\delta'} \subset V_{\delta''}$, respectively.
	\end{lem}
	\begin{lem}\label{lem-8}
		Let $0<J(u)<d$ for some $u^m \in W_0^{1,p}(\Omega)$. Let $\delta_1<\delta_2$ are two roots of equation $J(u)=d(\delta)$. Then the sign of $I_{\delta}(u)$ is unchangeable for $\delta_1<\delta<\delta_2$.
	\end{lem}
	\begin{proof}[Proof of Lemma \ref{lem-8}]
		The proof follows from contradiction argument. $||\nabla u^m||_{L^p(\Omega)} \neq 0$ comes from $J(u)>0$. If sign of $I_{\delta}(u)$ is changeable for $\delta_1<\delta <\delta_2$, then there exists $\delta^* \in (\delta_1,\delta_2)$ such that $I_{\delta^*}(u)=0$. Thus by the definition of $d(\delta)$ we have $J(u)\geq d(\delta^*)$ that contradicts
		\begin{equation*}
		J(u)=d(\delta_1)=d(\delta_2) < d(\delta^*),
		\end{equation*}
		this completes the proof.
	\end{proof}
	\section{Low initial energy $J(u_0)<d$}\label{sec-4}
	In this section, we study a global existence of problem \eqref{main_eqn_p>2}.
	
	\subsection{Invariant sets of solutions}
	\begin{thm}\label{thm-inv-sets}
		Let $f(u)$ satisfy $(H)$, $u_0^m(x)\in W_0^{1,p}(\Omega)$. Assume that $0<e<d$, then equation $d(\delta)=e$ has two roots $\delta_1<\delta_2$. Therefore, we formulate the following properties
		\begin{itemize}
			\item[$(i)$] All solutions of problem \eqref{main_eqn_p>2} with $J(u_0)=e$ belong to $W_{\delta}$ for $\delta_1 <\delta < \delta_2$, provided  $I(u_0)>0$.
			\item[$(ii)$] All solutions of problem \eqref{main_eqn_p>2} with $J(u_0)=e$ belong to $V_{\delta}$ for $\delta_1 <\delta < \delta_2$, provided  $I(u_0)<0$.
		\end{itemize}
	\end{thm}
	\begin{proof}[Proof of Theorem \ref{thm-inv-sets}]
		In the case $(i)$, we suppose that $u(x,t)$ is any solution of problem \eqref{main_eqn_p>2}  with $J(u_0)=e$ and $I(u_0)>0$. If $I(u_0)>0$, then we have
		\begin{equation}\label{eq-3.2}
		m\int_0^t \int_{\Omega} u^{m-1} u_{\tau}^2 dx d\tau + J(u) \leq J(u_0)= d(\delta_1)=d(\delta_2)< d(\delta), \,\,\, \delta_1 < \delta <\delta_2.
		\end{equation}
		From \eqref{eq-3.2} and the definition of $d(\delta)$ we get $I_{\delta}(u_0)>0$ and $J(u_0)<d(\delta)$ that is $u_0^m(x) \in W_{\delta}$ for $0<\delta<b$.
		
		Now we prove $u^m(t) \in W_{\delta}$ for $\delta_1<\delta<\delta_2$ and $0<t<T$, where $T$ is the maximal existence time of $u^m(t)$. Arguing by contradiction, there must exist $t_0 \in (0,T)$ such that $u^m(t_0)\in \partial W_{\delta_0}$ for some $\delta_0 \in (\delta_1, \delta_2)$, and $$I_{\delta_0}(u(t_0))=0\,\,\,\,\,\text{or}\,\,\,\,\,J(u(t_0))=d(\delta_0).$$ From \eqref{eq-3.2}
		it is easy to see that $J(u(t_0))= d(\delta_0)$ is impossible. If $I_{\delta_0}(u(t_0))=0$, then by the definition of $d(\delta_0)$ we have $J(u(t_0))\geq d(\delta_0)$ that contradicts \eqref{eq-3.2}.
		
		In the case $(ii)$, we assume again that $u(x,t)$ is any solution of problem \eqref{main_eqn_p>2} with $$J(u_0)=e\,\,\,\,\text{and}\,\,\,\,I(u_0)<0.$$ As a previous case, $u_0^m(x)\in V_{\delta}$ can be established by Lemma \ref{lem-8} and \eqref{eq-3.2}.
		
		The proof of $u^m(t) \in V_{\delta}$ for $\delta \in (\delta_1,\delta_2)$ follows from arguing by contradiction. Let $t_0 \in (0,T)$ be the first time such that $u^m(t) \in V_{\delta}$ for $t \in [0,t_0)$ and $u^m(t_0) \in \partial V_{\delta_0}$, i.e. $$I_{\delta_0}(u(t_0))=0\,\,\,\,\text{or}\,\,\,\,J(u(t_0))=d(\delta_0),$$ for some $\delta_0 \in (\delta_1,\delta_2)$. From \eqref{eq-3.2} it follows that $J(u(t_0))\neq d(\delta_0)$. If $I_{\delta_0}(u(t_0))=0$, then $I_{\delta_0}(u(t))<0$ for $t \in (0,t_0)$ and Lemma \ref{lem-4} yield $$|| \nabla u^m(t)||_{L^p(\Omega)}>r(\delta_0)\,\,\,\,\text{and}\,\,\,\,|| \nabla u^m(t_0)||_{L^p(\Omega)}\geq r(\delta_0).$$ Hence by the definition of $d(\delta_0)$ we have $J(u(t_0))\geq d(\delta_0)$ which contradicts \eqref{eq-3.2}.
	\end{proof}
	\subsection{Global existence of weak solutions}
	
	Now we construct a global weak solution of the initial-boundary problem \eqref{main_eqn_p>2} by the Galerkin approximation method.
	\begin{thm}\label{thm-GWS}(Global existence for $J(u_0)<d$).
		Let $f$ satisfy $(H)$, $u_0^m \in W_0^{1,p}(\Omega)$. Assume that $J(u_0)<d$ and $I(u_0)>0$. The initial-boundary value problem \eqref{main_eqn_p>2} admits a global weak solution $u^m(t)\in L^{\infty}(0,\infty; W_0^{1,p}(\Omega))$ along with $(u^{\frac{m+1}{2}})_t(t)\in L^{2}(0,\infty;L^2(\Omega))$ and $u^m(t)\in W$ for $t \in [0,\infty)$.
	\end{thm}
	\begin{proof}[Proof of Theorem \ref{thm-GWS}]
		We employ the Galerkin method by selecting smooth functions $\omega_j=(\omega_j(x))$ $(j=1,\ldots)$ such that for a fixed positive integer $k$
		\begin{equation*}
		u_k(t) := \sum_{j=1}^k g_{jk}(t)\omega_j(x), \,\,\, k=1,2,\ldots
		\end{equation*}
		satisfying
		\begin{equation}\label{eq-weak}
		( u_{k}', \omega_s^m) +  (|\nabla u^m_k|^{p-2}  \nabla u^m_{k}, \nabla \omega_s^m )  = ( f(u_k),\omega_s^m) \,\,\, (0\leq t <\infty, \,\, s=1,\ldots,k).
		\end{equation}
		\begin{equation}\label{eq-ic}
		u_k^m(x,0) = \sum_{j=1}^{k} a_{jk}^m\omega_j^m(x) \rightarrow u_0^m(x) \,\,\, \text{ in } \,\, W^{1,p}_0(\Omega).
		\end{equation}
		Multiplying \eqref{eq-weak} by $(g_{sk}^m)'(t)$ and integrating with respect to time from $0$ to $t$, which gives
		\begin{align*}
		m\int_0^t (u_{k}',  g_{sk}^{m-1}g_{sk}'\omega_s^m) d\tau& +  \int_0^t (|\nabla u^m_k|^{p-2}\nabla u^m_{k}, (g_{sk}^m)'\nabla \omega_s^m ) d\tau \\&= \int_0^t (f(u_k),(g_{sk}^m)'\omega_s^m )d\tau,
		\end{align*}
		for $(0\leq t <\infty, \,\, s=1,\ldots,k)$, then summing over $s$, we have
		\begin{align*}
		\int_0^t ( [u_{k}'(\tau)]^2, u^{m-1}_{k}(\tau)) d\tau& +\frac{1}{p} \int_0^t \frac{d}{d\tau} \int_{\Omega} |\nabla u^m_k(\tau)|^{p} dx d\tau\\& = \int_0^t \frac{d}{d\tau}  \int_{\Omega} F(u_k(\tau)) dxd\tau.
		\end{align*}
		That yields
		\begin{align*}
		\int_0^t( [u_{k}'(\tau)]^2, u^{m-1}_{k}(\tau)) d\tau + J(u_k(t)) = J(u_{k}(0)), \,\,\, t \in [0,\infty).
		\end{align*}
		From \eqref{eq-ic} we get $J(u_{k}(0))\rightarrow J(u_0)$, then for sufficiently large $k$, we have
		\begin{align}\label{eq-4.3}
		\int_0^t( [u_{k}'(\tau)]^2, u^{m-1}_{k}(\tau))  d\tau + J(u_k(t)) < d, \,\,\, t \in [0,\infty).
		\end{align}
		Combining \eqref{eq-4.3} and Theorem \ref{thm-inv-sets}, we obtain $u_k^m(t) \in W$ for $t \in [0,\infty)$ and sufficiently  large $k$. From \eqref{eq-4.3} and
		\begin{align*}
		J(u) &= \frac{1}{p} || \nabla u^m ||^p_{L^p(\Omega)} - \int_{\Omega } [F(u)- \sigma] dx\\
		&\geq  \frac{1}{p} || \nabla u^m ||^p_{L^p(\Omega)} - \frac{1}{\alpha} \int_{\Omega }u^mf(u)dx- \frac{\beta}{\alpha} \int_{\Omega }u^{pm}dx  \\
		& \geq \left[\frac{1}{p}  - \frac{1}{\alpha} - \frac{\beta}{\lambda_{1,p}\alpha} \right]|| \nabla u^m ||^p_{L^p(\Omega)}  + \frac{1}{\alpha}I(u),
		\end{align*}
		we establish
		\begin{align}\label{eq-4.4}
		\int_0^t( [u_{k}'(\tau)]^2, u^{m-1}_{k}(\tau))  d\tau  + \left[\frac{1}{p}  - \frac{1}{\alpha} - \frac{\beta}{\lambda_{1,p}\alpha} \right]|| \nabla u^m_k ||^p_{L^p(\Omega)}  < d, \,\,\, t \in [0,\infty)
		\end{align}
		for sufficiently large $k$.  Then for $0\leq t <\infty $ \eqref{eq-4.4} implies the following estimates
		\begin{align*}
		&|| \nabla u^m_k ||^p_{L^p(\Omega)}  < \left[\frac{1}{p}  - \frac{1}{\alpha} - \frac{\beta}{\lambda_{1,p}\alpha} \right]^{-1}d,\\
		&||u^m_k||^{\gamma}_{L^{\gamma}(\Omega)} \leq C_*^{\gamma} 	|| \nabla u^m_k ||_{L^p(\Omega)}^{\gamma} < C_*^{\gamma}\left[\frac{1}{p}  - \frac{1}{\alpha} - \frac{\beta}{\lambda_{1,p}\alpha} \right]^{-\frac{\gamma}{p}}d^{\frac{\gamma}{p}},\\
		&	\int_0^t ||(u_k^{\frac{m+1}{2}})_{\tau}||^2_{L^2(\Omega)}   d\tau  < \frac{(m+1)^2}{4}d,
		\end{align*}
		and
		\begin{align*}
		|| f(u_k)||^q_{L^q(\Omega)} &\leq \int_{\Omega } [\gamma A |u_k|^{\gamma-m}]^q dx = \gamma^q A^q ||  u_k^m||^{\gamma}_{L^{\gamma}(\Omega)} \nonumber \\
		& < \gamma^q A^q C_*^{\gamma}\left[\frac{1}{p}  - \frac{1}{\alpha} - \frac{\beta}{\lambda_{1,p}\alpha} \right]^{-\frac{\gamma}{p}}d^{\frac{\gamma}{p}},
		\end{align*}
		where $q = \frac{m\gamma}{\gamma -m}$.
		
		According the above energy estimates, we see that sequence $\{u_k^m \}_{k=1}^{\infty}$ is bounded in $L^{\infty}(0,\infty;W_0^{1,p}(\Omega))$ and $\{(u_k^{\frac{m+1}{2}})' \}_{k=1}^{\infty}$ is bounded in $L^{2}(0,\infty;L^2(\Omega))$.
		
		Consequently, there exists a subsequence $\{  u^m_{k_l}\}_{l=1}^{\infty} \subset \{u^m_k \}_{k=1}^{\infty}$ and the functions $u^m(t)\in L^{\infty}(0,\infty;W_0^{1,p}(\Omega)) $, with $(u^{\frac{m+1}{2}})'(t) \in L^{2}(0,\infty;L^2(\Omega))$ and $f(u) \in L^{\infty}(0,\infty;L^q(\Omega))$, such that
		\begin{align*}
		u_{k_l}^m &\rightarrow u^m \,\, \text{ weakly star in }  L^{\infty}(0,\infty;W_0^{1,p}(\Omega)) \,\, \text{ and a.e. } Q=\Omega \times [0,\infty) \\
		(u_{k_l}^{\frac{m+1}{2}})_t &\rightarrow (u^{\frac{m+1}{2}})_t \,\, \text{ weakly in }  L^{2}(0,\infty;L^2(\Omega)),\\
		f(u_{k_l})&\rightarrow f(u) \,\, \text{  weakly star in }  L^{\infty}(0,\infty;L^q(\Omega)) \,\,\, \text{ and a.e. } Q=\Omega \times [0,\infty).
		\end{align*}
		For fixed $s$ in \eqref{eq-weak} and taking $k=k_l \rightarrow \infty$ we get  	
		\begin{equation*}
		(	 u_t, \omega_s^m ) + (|\nabla u^m|^{p-2} \nabla u^m, \nabla \omega_s^m )\rangle dx  = (f(u),\omega_s^m), \,\, \text{  for all } \,\,s,
		\end{equation*}
		and
		\begin{equation*}
		\int_{\Omega} u_t v dx + \int_{\Omega} |\nabla u^m|^{p-2}\nabla u^m \cdot \nabla v dx = \int_{\Omega} f(u)v dx, \,\,\, \forall v \in W_0^{1,p}(\Omega), \,\, t \in [0,\infty),
		\end{equation*}
		\begin{equation*}
		u^m(x,0) = u_0^m(x) \,\,\, \text{ in } \,\, W_0^{1,p}(\Omega).
		\end{equation*}
		Therefore $u(x,t)$ is a global weak solution of problem \eqref{main_eqn_p>2}.
	\end{proof}
	\subsection{Asymptotic behavior of weak solutions}	
	Since Theorem \ref{thm-GWS} provides the existence of global weak solutions of initial-boundary problem \eqref{main_eqn_p>2}, we shall consider the asymptotic behavior of solutions for problem \eqref{main_eqn_p>2}.
	\begin{thm}\label{thm-global}
		Let $f(u)$ satisfy $(H)$ and $u_0^m(x)\in W_0^{1,p}(\Omega)$. Suppose that $0<J(u_0)<d$ and $I(u_0)>0$.
		\begin{description}
			\item[(i)] If $pm> m+1,$ then for a global weak solution $u(x,t)$ of problem \eqref{main_eqn_p>2}, there exists a positive constant $\lambda$ such that
			\begin{equation*}
			|| u(t)||_{L^{m+1}(\Omega)} \leq ( || u_0||_{L^{m+1}(\Omega)}^{m+1-pm} + \lambda t)^{- \frac{1}{pm-m-1}},\, 0\leq t < \infty.
			\end{equation*}
			\item[(ii)] If $pm=m+1,$ then the global weak solution $u(x,t)$ of problem \eqref{main_eqn_p>2} satisfies the estimate
			\begin{equation*}
			|| u(t)||_{L^{m+1}(\Omega)} \leq e^{-(m+1)(1-\delta_1)Ct}|| u_0||_{L^{m+1}(\Omega)},\,C>0,\,0\leq t<\infty.
			\end{equation*}
		\end{description}
	\end{thm}
	\begin{proof}[Proof of Theorem \ref{thm-global}]
		(i) Let $u(t) $ be any global weak solution of initial-boundary problem \eqref{main_eqn_p>2} with the initial conditions $J(u_0)<d$ and $I(u_0)>0$.
		Let us multiply \eqref{eq-weak-solution} by $d(t)\in [0,\infty)$, then we have
		\begin{align*}
		&	 (u_t, d(t)v) dx +(|\nabla u^m|^{p-2}\nabla u^m, \nabla d(t)v) =  (f(u),d(t)v ),
		\end{align*}
		for all $v(x) \in W_0^{1,p}(\Omega)$ and $ d(t) \in [0,\infty)$
		and write
		\begin{equation*}
		\int_{\Omega} u_t w dx + \int_{\Omega} |\nabla u^m|^{p-2}\nabla u^m \cdot \nabla w dx = \int_{\Omega} f(u)w dx, \,\,\, \forall w \in L^{\infty}(0,\infty; W_0^{1,p}(\Omega)).
		\end{equation*}
		Setting $w=u^m$, we arrive at
		\begin{equation*}
		\int_{\Omega }u_t u^m dx + \int_{\Omega } |\nabla u^m|^p dx = \int_{\Omega } u^mf(u)dx,
		\end{equation*}
		and it can be rewritten such as
		\begin{equation*}
		\frac{1}{m+1} \frac{d}{dt} \int_{\Omega }u^{m+1}dx + I(u) = 0.
		\end{equation*}
		Theorem \ref{thm-inv-sets} along with $$0<J(u_0)<d\,\,\,\,\text{and}\,\,\,\,I(u_0)>0$$ imply that  $u(t)\in W_{\delta}$ for $\delta \in (\delta_1,\delta_2)$ and $0\leq t < \infty$, where $\delta_1<\delta_2$ are the two roots of equation $d(\delta)=J(u_0)$. Hence, we have $$I_{\delta}(u)\geq 0\,\,\, \text{for} \,\,\,\delta \in (\delta_1,\delta_2)$$ and $$I_{\delta_1}(u)\geq 0\,\,\, \text{for} \,\,\,0\leq t < \infty.$$ Then we have
		\begin{align*}
		\frac{1}{m+1} \frac{d}{dt} \int_{\Omega }u^{m+1}dx + (1-\delta_1) \int_{\Omega} |\nabla u^m|^p dx + I_{\delta_1}(u)=0,
		\end{align*}
		and
		\begin{align*}
		\frac{1}{m+1} \frac{d}{dt} \int_{\Omega }u^{m+1}dx + (1-\delta_1) \int_{\Omega} |\nabla u^m|^p dx \leq 0.
		\end{align*}
		By the Sobolev inequality
		\begin{equation*}
		\int_{\Omega } |\nabla u^m|^p dx\geq C_*^p \left(\int_{\Omega } |u^{m}|^{p^*} dx\right)^{\frac{pm}{p^*m}},\,p^*=\frac{np}{n-p},
		\end{equation*}
		and embedding for $mp^*>m+1$
		\begin{equation*}
		|| u ||_{L^{mp^*}(\Omega)} \geq C || u ||_{L^{m+1}(\Omega)},
		\end{equation*}
		we obtain the following estimate
		\begin{equation*}
		\int_{\Omega } |\nabla u^m|^p dx\geq C_*^pC \left(\int_{\Omega } u^{m+1}dx \right)^{\frac{pm}{m+1}}.
		\end{equation*}
		That gives
		\begin{equation}\label{Assymp}
		\frac{1}{m+1} \frac{d}{dt} \int_{\Omega }u^{m+1}dx + (1-\delta_1)C_*^pC  \left(\int_{\Omega } u^{m+1}dx \right)^{\frac{pm}{m+1}} \leq 0.
		\end{equation}
		Therefore, there exists a constant $\lambda>0$ such that
		\begin{equation*}
		\int_{\Omega }u^{m+1}(t)dx \leq \left( \left[\int_{\Omega }u^{m+1}_0dx\right]^{1-s} + \lambda t \right)^{-\frac{1}{s-1}},
		\end{equation*}
		where $s = \frac{pm}{m+1}$ with $pm> m+1$ and
		\begin{equation*}
		\lambda := (m+1)(s-1)(1-\delta_1)C_*^pC>0.
		\end{equation*}
		
		The case (ii) immediately follows from the inequality \eqref{Assymp}, if we replace $pm$ by $m+1.$
		
		This completes the proof.
	\end{proof}
	\subsection{Finite time blow-up for $J(u_0)<d$}
	
	\begin{thm}\label{thm-blow}
		Let $f(u)$ satisfy $(H)$ and $u_0^m(x)\in W_0^{1,p}(\Omega)$. Suppose that $J(u_0)<d$ and $I(u_0)<0$. Then the solution of problem \eqref{main_eqn_p>2} blows up in a finite time, there exists time $T^*>0$ such that
		\begin{equation}
		\lim_{t\rightarrow T^*} \int_0^t \int_{\Omega } u^{m+1}(x, \tau) dx d\tau = + \infty.
		\end{equation}
	\end{thm}
	\begin{proof}[Proof of Theorem \ref{thm-blow}] The strategy of proof is to show the differential inequality
		\begin{equation}\label{eq-M}
		M''(t) M(t) - (1+\alpha) (M'(t))^2 > 0,
		\end{equation}
		where
		\begin{equation}
		M(t) = \int_0^t \int_{\Omega} u^{m+1}(x, \tau)dxd\tau + M_0,
		\end{equation}
		where $M_0$ is a positive constant.
		In order to prove \eqref{eq-M}, we need the following computations	
		\begin{align*}
		M'(t) &= \int_{\Omega} u^{m+1}(x, \tau)dx\end{align*} and
		\begin{align*}M''(t)&= (m+1)\int_{\Omega} [u^m f(u) - |\nabla u^m|^p] dx = -(m+1)I(u).
		\end{align*}
		Also, we have 	
		\begin{align*}
		M''(t) &=(m+1) \int_{\Omega} u^{m}(x,t) u_t(x,t) dx\\
		& = (m+1) \int_{\Omega} u^{m}(x,t)  \Delta_p u^m(x,t) + (m+1)\int_{\Omega} u^{m}(x,t)f(u(x, t))dx \\
		& = -(m+1) \int_{\Omega} |\nabla u^m(x,t)|^p dx + (m+1)\int_{\Omega} u^{m}(x,t)f(u(x,t))dx
		\\
		& \geq - (m+1) \int_{\Omega} |\nabla u^m(x,t)|^p dx + (m+1) \int_{\Omega} \left[ \alpha F(u(x,t)) -\beta u^{pm} (x,t) -\alpha \sigma \right] dx\\
		& =-\alpha (m+1) \left[ \frac{1}{p} \int_{\Omega} |\nabla u^m(x,t)|^p dx - \int_{\Omega} (F(u(x,t))-\sigma) dx \right]  \\
		&+ \left(\frac{\alpha(m+1)}{p}-m-1\right) \int_{\Omega} |\nabla u^m(x,t)|^p dx -\beta (m+1) \int_{\Omega} u^{pm}(x, t) dx\\
		& \geq -\alpha(m+1) J(u)  + \left[ \lambda_{1,p}\frac{(\alpha-p)(m+1)}{p}\ - \beta (m+1) \right]\int_{\Omega} u^{pm}(x, t) dx \\
		& \geq -\alpha(m+1) J(u_0) + \alpha m(m+1)\int_{0}^t \int_{\Omega} u^{m-1}(x, \tau) (u_{\tau}(x, \tau))^2 dx d\tau \\
		&+ \left[ \lambda_{1,p}\frac{(\alpha-p)(m+1)}{p}\ - \beta (m+1) \right]\int_{\Omega} u^{pm}(x, t) dx,
		\end{align*}
		and
		\begin{align*}
		(	M'(t) )^2 &\leq (m+1)^2(1+\delta^*) \left(\int_{0}^t \int_{\Omega} u^{m+1} (x, \tau)dx d\tau\right)\left( \int_{0}^t \int_{\Omega} u^{m-1}(x, \tau) u_{\tau}^2(x,\tau)dxd\tau \right)   \\
		&+ \left( 1+ \frac{1}{\delta^*}\right)\left( \int_{\Omega} u_0^{m+1}(x)dx \right)^2 \,\,\, \text{ for } \delta^*>0.
		\end{align*}
		Then we write
		\begin{align*}
		&M''(t)M(t) - (1+\delta^* ) (M'(t))^2 \geq -\alpha(m+1) M(t)J(u_0)  + C_1 M(t)\int_{\Omega} u^{pm}(x, t) dx\\
		&+ \alpha m(m+1)\int_{0}^t \int_{\Omega} u^{m-1}(x, \tau) (u_{\tau}(x, \tau))^2 dxd\tau \left( \int_{0}^t\int_{\Omega } u^{m+1}(x, \tau)dx d\tau + M_0 \right) \\
		&- (m+1)^2(1+\delta^*)^2 \left(\int_{0}^t \int_{\Omega} u^{m+1}(x, \tau) dx d\tau\right)\left( \int_{0}^t \int_{\Omega} u^{m-1}(x, \tau) u_{\tau}^2(x,\tau)dxd\tau \right) \\
		& - (1+\alpha)\left( 1+ \frac{1}{\delta^*}\right)\left( \int_{\Omega} u_0^{m+1}(x)dx \right)^2,
		\end{align*}
		where $C_1:= (m+1)(\lambda_{1,p}\frac{(\alpha-p)}{p}\ - \beta) $.
		Furthermore, by taking $\delta^* = \sqrt{\frac{\alpha m}{m+1}}-1>0$ the above expression yields
		\begin{align*}
		M''(t)M(t) &- (1+\alpha ) (M'(t))^2 \geq M(t)\left[\frac{ C_1}{2}\int_{\Omega} u^{pm}(x, t) dx -\alpha(m+1) J(u_0)\right] \\
		& +\frac{ C_1}{2} M(t)\int_{\Omega} u^{pm}(x, t) dx - (1+\alpha)\left( 1+ \frac{1}{\delta^*}\right)\left( \int_{\Omega} u_0^{m+1}(x)dx \right)^2.
		\end{align*}
		By the embedding for $pm \geq m+1$
		\begin{align*}
		|| u(t) ||_{L^{pm}(\Omega)} \geq || u(t)||_{L^{m+1}(\Omega)} = [M'(t)]^{\frac{1}{m+1}},
		\end{align*}
		we obtain
		\begin{align*}
		M''(t)M(t) &- (1+\alpha ) (M'(t))^2 \geq M(t)\left[\frac{ C_1}{2}[M'(t)]^{\frac{pm}{m+1}} dx -\alpha(m+1) J(u_0)\right] \\
		& +\frac{ C_1}{2} M(t)[M'(t)]^{\frac{pm}{m+1}} - (1+\alpha)\left( 1+ \frac{1}{\delta^*}\right)\left( \int_{\Omega} u_0^{m+1}(x)dx \right)^2.
		\end{align*}
		From the assumption $0<J(u_0)<d$, $I(u_0)<0$ and Theorem \ref{thm-inv-sets} we obtain $u^m \in V_{\delta}$ for $1<\delta < \delta_2$ and $t>0$, where $\delta_2$ is the larger root of equation $d(\delta)=J(u_0)$. Hence $I_{\delta}(u)<0$ and Lemma \ref{lem-4} (ii), we have $||\nabla u^m||_{L^p(\Omega)} > r(\delta)$ for $1<\delta < \delta_2$ and $t>0$. So we get $I_{\delta_2}(u)<0$ and $|| \nabla u^m ||_{L^p(\Omega)}> r(\delta_2)$ for $t>0$ and
		\begin{align*}
		M''(t) &= -(m+1)I(u) = (m+1)(\delta_2 - 1)||\nabla u^m||^p_{L^p(\Omega)} - (m+1)I_{\delta_2}(u) \\
		& \geq (m+1) (\delta_2 -1)r^p(\delta_2) > 0 \,\,\, \text{ for } \,\, t>0,\\
		M'(t) &\geq (m+1) r^p(\delta_2) t + M'(0) \geq (m+1) r^p(\delta_2) t\,\,\, \text{ for } \,\, t>0,\\
		M(t) &\geq (m+1) r^p(\delta_2) t^2 + M_0\,\,\, \text{ for } \,\, t>0.
		\end{align*} 
		From the above we make sure that $M'(t)$ increases, therefore there is $t_1$ such that for all $t\geq t_1$ the following inequalities hold
		\begin{equation*}
		\frac{C_1}{2} [M'(t)]^{\frac{pm}{m+1}} >	\alpha(m+1) J(u_0),
		\end{equation*}
		and
		\begin{equation*}
		\frac{C_1}{2} M(t) [M'(t)]^{\frac{pm}{m+1}} dx > (1+\alpha)\left( 1+ \frac{1}{\delta^*}\right)\left( \int_{\Omega} u_0^{m+1}(x)dx \right)^2.
		\end{equation*}
		Hence, we obtain the differential inequality 
		\begin{equation*}
		M''(t) M(t) - (1+\alpha) (M'(t))^2 > 0.
		\end{equation*}
		We can see that the above expression for $t\geq 0$ implies
		\begin{equation*}
		\frac{d}{dt} \left[ \frac{M'(t)}{M^{\alpha+1}(t)} \right] >0,\end{equation*} hence $$\begin{cases}
		M'(t) \geq \left[ \frac{M'(0)}{M^{\alpha+1}(0)} \right] M^{1+\alpha}(t),\\
		M(0)= M_0.
		\end{cases}$$
		Then we arrive at
		\begin{equation*}
		M(t) \geq \left( \frac{1}{M^{\alpha}_0}-\frac{ \alpha \int_{\Omega} u^{m+1}_0(x)dx }{M^{\alpha+1}_0} t\right)^{-\frac{1}{\alpha}}.
		\end{equation*}
		Then the blow-up time $T^*$ satisfies
		\begin{equation*}
		0<T^*\leq \frac{M_0}{\alpha \int_{\Omega} u_0^{m+1}dx}.
		\end{equation*}
		The proof is finished.
	\end{proof}
	\section{Critical initial energy $J(u_0)=d$}\label{sec-5}
	\subsection{Global existence for $J(u_0)=d$}
	\begin{thm}\label{thm-GWS-d}(Global existence for $J(u_0)=d$).
		Let $f$ satisfy $(H)$, $u_0^m \in W_0^{1,p}(\Omega)$. Assume that $J(u_0)=d$ and $I(u_0)\geq 0$. The initial-boundary value problem \eqref{main_eqn_p>2} admits a global weak solution $u^m(t)\in L^{\infty}(0,\infty; W_0^{1,p}(\Omega))$ along with $(u^{\frac{m+1}{2}})_t(t)\in L^{2}(0,\infty;L^2(\Omega))$ and $u(t)\in\overline{W}=W\cup \partial W$ for $t \in [0,\infty)$.
	\end{thm}
	\begin{proof}[Proof of Theorem \ref{thm-GWS-d}]
		Note that $J(u_0)=d$ implies $|| \nabla u_0^m||_{L^p(\Omega)}\neq 0$. Let us pick a sequence $\{\epsilon_k\}_{k=1}^{\infty}$ such that $0<\epsilon_k<1$ for $k=1,\ldots$ and $\epsilon_k\rightarrow 1$ as $k\rightarrow \infty$.  We consider the following initial-boundary problem
		\begin{align}\label{eq-J=d}
		\begin{cases}
		u_t - \nabla \cdot ( |\nabla u^m|^{p-2}\nabla u^m) = f(u), \,\,\, & (x,t) \in \Omega \times (0,+\infty), \\
		u(x,t)  =0,  \,\,\,& (x,t) \in \partial \Omega \times [0,+\infty), \\
		u(x,0)  = u_{0k}(x)\geq 0,\,\,\, & x \in \overline{\Omega},
		\end{cases}
		\end{align}
		where $u_{0k}(x) = \epsilon_ku_0(x)$ for $k=1,2,\ldots$. From $I(u_0)\geq 0$ and Lemma \ref{lem-3}, we have $\epsilon^*= \epsilon^{*}(u_0)\geq 1$. This gives
		\begin{equation*}
		I(u_{0k})= I(\epsilon_k u_0) >0\end{equation*} and \begin{equation*}J(u_{0k})= J(\epsilon_k u_0) <J(u_0)=d.
		\end{equation*}
		From Theorem \ref{thm-GWS} we know that for each $k$ problem \ref{eq-J=d} admits a global weak solution $u_{k}(t)\in L^{\infty}(0,\infty;W_0^{1,p}(\Omega))$ with $u_{k}'(t)\in L^{2}(0,\infty;L^2(\Omega))$ and $u_k(t)\in W$ for $t\in [0,\infty)$ satisfying
		\begin{equation*}
		(u_k\,v) + (|\nabla u^m_k|^{p-2}\nabla u_k^m , \nabla v) = (f(u_k),v), \,\,\,\, \forall v \in W_0^{1,p}(\Omega), t>0,
		\end{equation*}
		and
		\begin{equation*}
		\int_0^t( [u_{k}'(\tau)]^2, u^{m-1}_{k}(\tau))  d\tau + J(u_k) \leq J(u_{0k}) <d, \,\,\, t \in [0,\infty).
		\end{equation*}
		Then we establish
		\begin{align}
		\int_0^t( [u_{k}'(\tau)]^2, u^{m-1}_{k}(\tau))  d\tau  + \left[\frac{1}{p}  - \frac{1}{\alpha} - \frac{\beta}{\lambda_{1,p}\alpha} \right]|| \nabla u^m_k ||^p_{L^p(\Omega)}  < d, \,\,\, t \in [0,\infty)
		\end{align}
		for sufficiently large $k$.
		The remaining steps of the proof is the same way as in Theorem \ref{thm-GWS}.
	\end{proof}
	\subsection{Asymptotic behavior of solutions}	
	Since Theorem \ref{thm-GWS-d} provides the existence of global weak solutions of initial-boundary problem \eqref{main_eqn_p>2} for critical energy cindition $J(u_0)=d$, we shall consider the asymptotic behavior of solutions for problem \eqref{main_eqn_p>2}.
	\begin{thm}\label{thm-asym}
		Let $f(u)$ satisfy $(H)$ and $u_0^m(x)\in W_0^{1,p}(\Omega)$. Suppose that $J(u_0)=d$ and $I(u_0)\geq0$.
		\begin{description}
			\item[(i)] If $pm> m+1,$ then for a global weak solution $u(x,t)$ of problem \eqref{main_eqn_p>2}, there exists a positive constant $\lambda$ such that
			\begin{equation*}
			|| u(t)||_{L^{m+1}(\Omega)} \leq ( || u_0||_{L^{m+1}(\Omega)}^{m+1-pm} + \lambda t)^{- \frac{1}{pm-m-1}},\, t_1\leq t < \infty,\,t_1>0.
			\end{equation*}
			\item[(ii)] If $pm=m+1,$ then the global weak solution $u(x,t)$ of problem \eqref{main_eqn_p>2} satisfies the estimate
			\begin{equation*}
			|| u(t)||_{L^{m+1}(\Omega)} \leq e^{-(m+1)(1-\delta_1)Ct}|| u_0||_{L^{m+1}(\Omega)},\,C>0,\,t_1\leq t<\infty.
			\end{equation*}
		\end{description}
	\end{thm}
	\begin{proof}[Proof of Theorem \ref{thm-asym}]
		Recall that the existence of a global weak solution of problem \eqref{main_eqn_p>2} is given in Theorem \ref{thm-GWS-d} with $J(u_0)=d$ and $I(u_0)>0$, which implies $I(u)\geq 0$ for $0\leq t<\infty$. Now we consider the following two cases:
		
		$(a)$ Assume that $I(u)>0$ for $0\leq t<\infty$. From $(u_t,u^m)= -I(u)<0$ we have $\int_{\Omega } u_t^2 u^{m-1}dx>0$, it follows that $\int_0^t(u^2_{\tau},u^{m-1})d\tau$ is increasing for $0\leq t<\infty$. Picking any $t_1>0$ and letting
		\begin{equation*}
		d_1 = d - \int_{0}^{t_1} (u^2_{\tau},u^{m-1})d\tau.
		\end{equation*}
		Since we have
		\begin{align*}
		J(u) \leq J(u_0) - \int_{0}^{t_1} (u^2_{\tau},u^{m-1})d\tau= d_1<d,
		\end{align*}
		and $u^m(t)\in W_{\delta}$ for $\delta_1 < \delta< \delta_2$ and $t_1\leq t<\infty$, where $\delta_1<\delta_2$ are two roots of equation $d(\delta)=d_1$. Then we have $I_{\delta_1}(u)\geq 0$ for $t\geq t_1$ which gives
		\begin{align*}
		\frac{1}{m+1} \frac{d}{dt} \int_{\Omega }u^{m+1}dx + (1-\delta_1) \int_{\Omega} |\nabla u^m|^p dx + I_{\delta_1}(u)=0,
		\end{align*}
		and
		\begin{align*}
		\frac{1}{m+1} \frac{d}{dt} \int_{\Omega }u^{m+1}dx + (1-\delta_1) \int_{\Omega} |\nabla u^m|^p dx \leq 0.
		\end{align*}
		for $t_1\leq t<\infty$. By making use of following estimate
		\begin{equation*}
		\int_{\Omega } |\nabla u^m|^p dx\geq C_*^pC \left(\int_{\Omega } u^{m+1}dx \right)^{\frac{pm}{m+1}},
		\end{equation*}
		we have
		\begin{align*}
		\frac{1}{m+1} \frac{d}{dt} \int_{\Omega }u^{m+1}dx + (1-\delta_1)C_*^pC  \left(\int_{\Omega } u^{m+1}dx \right)^{\frac{pm}{m+1}} \leq 0.
		\end{align*}
		Therefore, for $t_1\leq t<\infty$, there exists a constant $\lambda>0$ such that
		\begin{equation*}
		\int_{\Omega }u^{m+1}(t)dx \leq \left( \left[\int_{\Omega }u^{m+1}_0dx\right]^{1-s} + \lambda t \right)^{-\frac{1}{s-1}},
		\end{equation*}
		where $s = \frac{pm}{m+1}$ with $pm> m+1$ and
		\begin{equation*}
		\lambda := (m+1)(s-1)(1-\delta_1)C_*^pC>0.
		\end{equation*}
		
		$(b)$ Assume that there exists $t_1>0$ such that $I(u(t_1))=0$ and $I(u(t))>0$ for $0\leq t<t_1$. Then $\int_0^t(u^2_{\tau},u^{m-1})d\tau$ is increasing for $0\leq t<t_1$. Then we have
		\begin{equation*}
		J(u(t_1)) \leq d - \int_{0}^{t_1} (u^2_{\tau},u^{m-1})d\tau <d,
		\end{equation*}
		it follows that $|| \nabla u^m ||_{L^p(\Omega)}=0$. Then we get
		\begin{align*}
		\frac{1}{m+1} \frac{d}{dt} \int_{\Omega }u^{m+1}dx + (1-\delta_1) \int_{\Omega} |\nabla u^m|^p dx \leq 0,
		\end{align*}
		for $t_1\leq t<\infty$. The rest of proof is same as the previous case.
		The case (ii) can be proved similarly.
		This completes the proof of Theorem \ref{thm-asym}.
	\end{proof}
	\subsection{Finite time blow-up for $J(u_0)=d$}
	\begin{thm}\label{thm-blow-d}
		Let $f(u)$ satisfy $(H)$ and $u_0^m(x)\in W_0^{1,p}(\Omega)$. Suppose that $J(u_0)=d$ and $I(u_0)<0$. Then the  solution of problem \eqref{main_eqn_p>2} blows up in a finite time, there exists time $T>0$ such that
		\begin{equation}
		\lim_{t\rightarrow T} \int_0^t \int_{\Omega } u^{m+1} dx d\tau = + \infty.
		\end{equation}
	\end{thm}
	\begin{proof}[Proof of Theorem \ref{thm-blow-d}] The proof is similar as in the case $J(u_0)< d$.  
	\end{proof}
	
	\section*{Conclusion and open questions}
	In this paper, we obtained some results on global existence, asymptotic behavior and blow-up of solutions for the porous medium equation \eqref{main_eqn_p>2} under condition $0<J(u_0)\leq d$ or $J(u_0)<0$.
	For convenience, these cases are given in the table below
	\begin{center}
		\begin{tabular}{||c c c||}
			\hline
			& "+" global existence & "-" blow-up \\ [0.5ex]
			\hline\hline
			$J(u_0)< 0$ & - & - \\
			\hline
			$0<J(u_0)<d$  & + & - \\
			\hline
			$J(u_0)=d$ & + & - \\
			\hline
			$J(u_0)>d$ & ? & ? \\
			\hline
			& $I(u_0)\geq 0$ & $I(u_0)< 0$ \\ [1ex]
			\hline
		\end{tabular}
	\end{center}
	It is easy to notice from the above table that in supercritical initial energy case $J(u_0)>d$ the problem \eqref{main_eqn_p>2} has not yet been investigated. It should be noted that such cases were studied in \cite{Xu2} for the pseudoparabolic equation with power nonlinearity.
	
	All results for this paper were obtained for the porous medium equation \eqref{main_eqn_p>2} with $m\geq 1$. That is, the fast diffusion equation case with $0<m<1$ is still open.
	
	The above reasoning allows us to state with confidence about the future perspective of the problem \eqref{main_eqn_p>2}.
	
	%\section*{Data Availability} Data sharing not applicable to this article as no datasets were generated or analysed during the current study
	
	%\section*{Declaration of competing interest} 
	%The Authors declares that there is no conflict of interest
	
	\section*{Acknowledgments}
	The authors would like to thank the reviewers for their valuable comments and remarks.

\end{document}